\documentclass[12pt]{amsart}
\usepackage[english]{babel}
\usepackage{amsmath}
\usepackage{changes}
\usepackage{amsthm}
\usepackage{amsfonts} 
\usepackage{hyperref}
\usepackage{epsfig}
\usepackage{amssymb}
\usepackage{mathrsfs}
\DeclareMathOperator*{\argmin}{arg\,min}

\numberwithin{equation}{section}

\allowdisplaybreaks

\def\parent#1{#1^{-1}}
\def\r{\varrho}

\newcommand{\eqa}{\begin{eqnarray}}
\newcommand{\ena}{\end{eqnarray}}
\newcommand{\eq}{\begin{equation}}
\newcommand{\en}{\end{equation}}
\newcommand{\eqs}{\begin{eqnarray*}}
\newcommand{\ens}{\end{eqnarray*}}

\def\X{\mathbf{X}} 

\def\n{\nu} 
\def\r{\varrho} 

\newcommand{\Z}     {\mathbb{Z}} 
\newcommand{\N}     {\mathbb{N}}

\newcommand{\Go}     {{\scriptsize GORW}}
\newcommand{\Or}     {{\scriptsize ORRW}}
\newcommand{\Er}     {{\scriptsize ERRW}}
\newcommand{\Rw}     {{\scriptsize RWRE}}

\def\1{{\mathchoice {1\mskip-4mu\mathrm l}      
{1\mskip-4mu\mathrm l} 
{1\mskip-4.5mu\mathrm l} {1\mskip-5mu\mathrm l}}} 
\newcommand{\ssup}[1] {{{\scriptscriptstyle{({#1}})}}} 
\def\comment#1{} 
 
\newcommand{\Ccal}   {{\mathcal C }}

\newcommand{\Gcal}   {{\mathcal G }}

\def\ignore#1{}

\def\Def{\ :=\ }

\def\C{\mathcal{C}}

\def\bP{\mathbf{P}}

\def\parent#1{#1^{-1}}

\newtheorem{theorem}{Theorem}
\newtheorem{proposition}[theorem]{Proposition}
\newtheorem{lemma}[theorem]{Lemma}
\newtheorem{remark}[theorem]{Remark}

\newtheorem{definition}[theorem]{Definition}
\newtheorem{corollary}[theorem]{Corollary}
\newtheorem*{ack}{Acknowledgement}

\renewcommand{\epsilon}{\varepsilon}

\title[The branching-ruin number]{The branching-ruin number and the critical parameter of once-reinforced random walk on trees
 }
\date{}
\author[A.~Collevecchio]{Andrea Collevecchio}
\address{Andrea Collevecchio\\ School of Mathematical Sciences, Monash  University, Melbourne} \email{andrea.Collevecchio@monash.edu}
\author[D.~Kious]{Daniel Kious}
\address{Daniel Kious\\ NYU-ECNU Institute of Mathematical Sciences at NYU Shanghai} \email{daniel.kious@nyu.edu}
\author[V.~Sidoravicius]{Vladas Sidoravicius}
\address{Vladas Sidoravicius\\ Courant Institute of Mathematical Sciences, New York, and NYU-ECNU Institute of Mathematical Sciences at NYU Shanghai} \email{vs1138@nyu.edu}


\begin{document}

\begin{abstract}
The motivation for this paper
 is the study of the phase transition for recurrence/transience of a class of self-interacting random walks on trees, which includes the once-reinforced random walk.
For this purpose, we define a  quantity, that we call the {\it branching-ruin number} of a tree, which provides (in the spirit of Furstenberg \cite{Furs} and  Lyons \cite{L90}) a natural way to measure trees with polynomial growth. We prove that the branching-ruin number of a tree is equal to the critical parameter for the recurrence/transience of the once-reinforced random walk. We define a sharp and effective (i.e.~computable) criterion characterizing the recurrence/transience of a larger class of self-interacting walks on trees, providing the complete picture for their phase transition.
\end{abstract}

\maketitle

In this paper we  study  the phase transition for recurrence/transience of a class of self-interacting random walks on trees.
Our main tool is a   quantity, that we call the {\it branching-ruin number} of a tree, which provides a natural way to measure trees with polynomial growth. In particular, we prove that the branching-ruin number of a tree is equal to the critical parameter for the recurrence/transience of the once-reinforced random walk ({\scriptsize ORRW}) on this tree, providing the complete picture of its phase transition.
The last statement is a corollary of a more general study of a larger class of self-interacting random walks, for which we prove a sharp  and effective (i.e.~computable) criterion characterizing their recurrence or transience. This class of processes includes a generalization of the {\scriptsize ORRW}, as well as biased random walks, or random walks in random environment, see Remark \ref{remrwre}.

The study of self-interacting random walks is challenging, as they are not Markovian, and proving recurrence or transience is difficult.
Our approach provides  the first general technique for the study of  \Or.\\

The idea of the branching-ruin number stems  both from the Hausdorff dimension of a tree defined  by  Furstenberg \cite{Furs} and from the branching number introduced by Lyons \cite{L90} who linked it to biased random walks, percolation and Ising model on trees. In \cite{LP}, Lyons and Peres write ``{\it the branching number of a tree is a single number that captures enough of the complexity of a general tree to give the critical value for a stochastic process on the tree}''. The branching-ruin number aims at fulfilling the same mission, but for a different class of random walks and trees. The branching number is adapted to the study of trees with exponential growth. The branching-ruin number is designed for the study of trees with polynomial growth (see Section \ref{features}) and is strikingly related to the critical parameter of the \Or.\\

The \Or~was introduced in 1990 by Davis \cite{Dav90}. Despite its  simple definition, the \Or~turns out to be difficult to analyze and, so far, no general tools were available for its study. The last author conjectured that on $\mathbb{Z}^d$, $d\ge3$, the {\scriptsize ORRW} undergoes a phase transition  recurrence/transience with respect to the reinforcement parameter. 
This problem is still open on the hypercubic lattice. In the two-dimensional case, recurrence on $\mathbb{Z}^2$ remains unsolved. 

Durrett, Kesten and Limic \cite{DKL02} proved that this conjecture does not hold on the binary tree and that {\scriptsize ORRW}  is transient for any choice of parameter. This was extended to supercritical Galton-Watson trees in \cite{Coll} (see also \cite{CHK} where  the positivity and the monotonicity of the speed on Galton-Watson trees is studied). Some partial results on ladders \cite{Sellke,Vervoort} are also available.  

Recently, the authors in \cite{KS16} provided  the first example of phase transition for  {\scriptsize ORRW} on $\mathbb{Z}^d$-like trees. It should be noted that these trees were spherically symmetric  with  a particular structure.

We should mention that a similar phase transition  was conjectured  for linearly edge-reinforced random walks (\Er) on $\mathbb{Z}^d$ in the eighties \cite{CD}, and was  first proved on regular trees in \cite{Pemtree}. Only recently, the phase transition recurrence/transience on $\mathbb{Z}^d$ , $d\ge3$, was established  in \cite{ST,ACK,DST}, see also \cite{SZ}.   However, techniques developed for  \Er~do not apply to \Or, in particular because exchangeability does not hold.

Here, we treat the case of  general trees. In particular, we recover and generalize any known result about \Or~by computing the branching-ruin number of the trees in these contexts, see Theorem \ref{corpoly}, Corollary \ref{corsuper} and Remark \ref{remzdlike}.
Besides, the sharp criterion in Theorem \ref{mainth} is stronger than existing results in the sense that it allows inhomogeneous initial weights and inhomogeneous reinforcement.\\
Finally, the main idea of our proof of transience relies on the presence of an infinite cluster for  a particular correlated percolation.

\section{The model}\label{sectionmodel}
\subsection{Notation} Let  $\Gcal=(V,E)$ be an infinite, locally finite, rooted tree with set of vertices $V$ and set of edges $E$. Let $\r$ be the root of $\mathcal{G}$.\\
For any vertex $\nu \in V\setminus\{{\r}\}$, denote  by $ {\nu}^{-1}$ its parent, i.e.  the neighbour of $\nu$ with shortest distance from $\r$.
For any $\nu \in V$, let  $|\nu|$ be  the number of edges in the unique self-avoiding path connecting $\nu$ to~$\r$ and call $|\nu|$ the {\it generation} of $\nu$. In particular, we have  $|\r|=0$.  For any edge $e \in E$
denote by  $e^-$ and $e^+$ its endpoints with $|e^+|=|e^-|+1$, and define the generation of an edge as $|e|=|e^+|$.\\
Two vertices $\nu,\mu\in V$ are called {\it neighbors}, denoted $\nu\sim\mu$, if they are the endpoints of a given edge $e$, that is $\{\mu,\nu\}=\{e^-,e^+\}$.\\
For any pair of  vertices $\nu$ and $\mu$, we write $\nu \le \mu$ if $\nu$ is on the unique self-avoiding path between ${\r}$ and $\mu$ (including it), and $\nu< \mu$ if moreover $\nu \neq \mu$. Similarly, {for two edges $e$ and $g$, we write $g\le e$ if $g^+ \le e^+ $ and $g<e$ if moreover $g^+\neq e^+$}. {For two vertices $\nu<\mu\in V$, we will denote by $[\nu,\mu]$  the unique self-avoiding path connecting $\nu$ to $\mu$.} For two neighboring vertices $\nu$ and $\mu$, we use the slight abuse of notation $[\nu,\mu]$ to denote the edge with endpoints $\nu$ and $\mu$ (note that we allow $\mu<\nu$).\\
 For two edges $e_1,e_2\in E$, we denote $e_1\wedge e_2$ the vertex with maximal distance from $\r$ such that $e_1\wedge e_2\le e_1^+$ and $e_1\wedge e_2\le e_2^+$.\\

\subsection{Definition of the model}\label{secdefmodel}
We define a generalized version of the Once-reinforced random walks, that we  denote by  \Go. This  process, denoted by $\X=(X_n)_n$, is discrete-time  and  takes values on the vertices of the tree $\Gcal$.  It starts from $\r$,  i.e. $X_0=\r$.  At each step, it jumps to one of the neighbors of the present state, according to the  rule described below. To any edge $e\in E$, we associate an {\it initial weight} $w_e\in(0,\infty)$ and a {\it reinforced weight} $\delta_e\in(0,\infty)$. 
Any edge is assigned its initial weight as long as it has not been crossed. After an edge is crossed for the first time,  it is assigned its reinforced weight from this time on (hence the weight of an edge is updated at most {\it once} in its whole life).
At each stage the walk   jumps through an edge with a probability that is proportional to its current weight.\\
More formally, let  $E_n$ be the collection of  edges crossed up to time $n$,  that is
\begin{align}\label{defEn}
E_n:=\left\{e\in E: \exists  k\in \{1,\dots,n\} \text{ s.t.~}\{X_{k-1},X_k\}=\{e^-,e^+\}\right\}.
\end{align}
At time $n\in\mathbb{N}$ and on the event $\left\{X_n=\nu \right\}$ with $\nu\in V$,  the walk jumps  to a neighbor $\mu\sim \nu$ with conditional probability
\[
\bP\left(\left.X_{n+1}=\mu\right|\mathcal{F}_n\right)=\frac{\delta_{[\nu,\mu]}\1_{[\nu,\mu]\in E_n}+w_{[\nu,\mu]}\1_{[\nu,\mu]\notin E_n}}{\sum_{\mu':\mu'\sim \nu}\left(\delta_{[\nu,\mu']}\1_{[\nu,\mu']\in E_n}+w_{[\nu,\mu']}\1_{[\nu,\mu']\notin E_n}\right)},
\]
where $\left(\mathcal{F}_n\right)$ is the natural filtration generated by the history of $\X$, i.e.~$\mathcal{F}_n=\sigma(X_k,0\le k\le n)$ for any integer $n\ge0$. The case when $w_e=1$ and $\delta_e=\delta$ for any $e\in E$ and for some $\delta\in(0,\infty)$ corresponds to the Once edge-reinforced random walk ({\scriptsize ORRW}) with parameter $\delta$. Note that the model we defined includes usual reversible Markov chains on trees, as well as various generalized versions of the {\scriptsize ORRW} (see  \cite{CHK} for instance).\\

A \Go~is said to be {\it recurrent} if, $\bP$-a.s.,  it eventually returns to ${\r}$. This process is {\it transient} if it is not recurrent, i.e.
$$ \bP\Big(T({\r}) =\infty\Big) >0,$$
where, for a vertex $v\in V$, $T(v)$ stands for the {\it return time} to $v$, that is
\[
T(v) \Def \inf\{n>0:X_n=v\}.
\]
In Section \ref{dichotomy}, we  prove a 0-1 law implying the equivalence between transience (resp.~recurrence) and the fact that \Go~visits each vertex finitely (resp.~infinitely) often almost surely.

\section{Main results}
\subsection{The branching-ruin number and the {\scriptsize ORRW}}
Let us fix an infinite, locally finite, rooted tree $\mathcal{G}$. Our first goal is to define the {\it branching-ruin number} of  $\mathcal{G}$.\\
We will need the notion of {\it cutsets}. A cutset  is a set $\pi$ of edges such that, for any infinite self-avoiding path $(\nu_i)_{i\ge0}$ started at the root,  there exists a unique $i\ge0$ such that $[\nu_{i-1},\nu_i]\in \pi$. In other words, a cutset is a minimal set of edges separating the root from infinity. We use $\Pi$ to denote the set of cutsets.\\
The branching-ruin number of the tree $\mathcal{G}$ is defined as
 \begin{align}\label{defBRNS}
br_r(\mathcal{G})= \sup\left\{\lambda>0: \inf_{\pi\in \Pi}\sum_{e\in\pi}|e|^{-\lambda}>0\right\}.
\end{align}
The branching-ruin number is  intrinsic to the  tree and is defined for any tree. Nevertheless, this quantity is particularly interesting when  measuring trees with polynomial growth. We give explanations and motivations for this fact in Section \ref{features}. It is an effective quantity in the sense that, in most cases, we can compute its value for a given tree. It is worth noting that, under some assumptions such as spherical symmetry (i.e.~the vertices within the same generation have the same number of children), a tree whose generation sizes grow like $n^b$ has a branching-ruin number equal to $b$. Also, a tree with subpolynomial growth has a branching-ruin number equal to $0$ and a spherically symmetric tree with exponential growth has an infinite branching-ruin number (see Corollary \ref{corsuper}).\\
Strikingly, the branching-ruin number of a tree is equal to the critical parameter for the recurrence/transience of the {\scriptsize ORRW} on this tree.\\

Recall that a random walk  ${\bf X}$  is {\scriptsize ORRW} with {\it reinforced parameter} $\delta\in(0,\infty)$ if it is a \Go, defined in Section \ref{secdefmodel}, with initial weights $w_e=1$ and reinforced weights $\delta_e=\delta$ for any edge $e$ of the tree.

The following theorem provides the full picture about  recurrence/transience of the {\scriptsize ORRW} on trees and identifies the value of the critical parameter.

\begin{theorem}\label{corpoly}
Fix an infinite, locally finite, tree $\mathcal{G}$ and let $br_r(\mathcal{G})\in[0,\infty]$ be its branching-ruin number. 
The {\scriptsize ORRW} with reinforced parameter $\delta\in(0,\infty)$ is transient if $\delta<br_r(\mathcal{G})$ and  recurrent if $\delta>br_r(\mathcal{G})$.
\end{theorem}

\begin{remark}\label{remzdlike}
In \cite{KS16}, two of the authors studied the {\scriptsize ORRW} on  $\mathbb{Z}^d$-like trees $\mathbb{T}_d$ whose vertices have $d$ children if they are at some generation $2^k$, $k\in\mathbb{N}$, and only one child otherwise. One can easily compute that these trees  have a branching-ruin number $br_r(\mathbb{T}_d)=\log_2(d)$ and thus recover the result from \cite{KS16} using Theorem \ref{corpoly}.
\end{remark}

In some situations, we are able to describe the behavior at criticality. The next result is proved in Section \ref{sectcrit}.

\begin{proposition}\label{propcrit}
Fix an infinite, locally finite tree $\mathcal{G}$ and consider the  {\scriptsize ORRW} $\X$ with critical parameter  $\delta_c=br_r(\mathcal{G})\in(0,\infty)$. 
First, if
\[
\inf_{\pi \in\Pi} \sum_{e\in \pi}{ |e|^{-\delta_c}}=0,
\]
then $\X$ is recurrent. Second, if there exists a positive function $f$ such that
\[
\inf_{\pi \in\Pi} \sum_{e\in \pi}\frac{1}{ |e|^{\delta_c}f(|e|)}>0\text{ and } \sum_{n\ge1}\frac{1}{nf(n)}<\infty,
\]
then $\X$ is transient.
\end{proposition}

In the light of the last result, one can easily show for instance that on a spherically symmetric tree that grows like $n^a/\log(n)$, the critical {\scriptsize ORRW} is recurrent, whereas if the tree grows like $n^a \log^2(n)$ then it is transient.\\

As mentioned in the introduction, the branching-ruin number is related to the branching number of the tree, studied by R.~Lyons \cite{L90} and defined as
\begin{align}\label{defbranch}
br(\mathcal{G})\Def \sup\left\{\lambda>0: \inf_{\pi\in \Pi}\sum_{e\in\pi}\lambda^{-|e|}>0\right\}.
\end{align}
Let us recall that any regular tree and any supercritical Galton-Watson tree, on the event of non-extinction,  has a branching number a.s.~equal to its mean offspring and thus strictly larger than $1$. Therefore, the following simple consequence of Theorem \ref{corpoly} generalizes  results of    Durrett, Kesten and Limic \cite{DKL02} and  results in   \cite{Coll}.
\begin{corollary}\label{corsuper}
Consider \Or~with parameter $\delta$ defined on a tree $\mathcal{G}$ which satisfies $br(\mathcal{G})>1$, where $br(\mathcal{G})$ is the branching number defined in \eqref{defbranch}.
This process  is transient for any $\delta\in(0,\infty)$.
 \end{corollary}

In Section \ref{features}, we present other interesting examples of trees with polynomial growth and compute their branching-ruin numbers.

\subsection{The sharp criterion}
Let us now state our most general result, which is a sharp and effective criterion for the recurrence/transience of  \Go, deeply related to the branching-ruin number.\\

Let us now consider  \Go~${\bf X}$, defined as in Section \ref{secdefmodel}, with initial weights $(w_e)_{e\in E}$ and reinforced weights $(\delta_e)_{e\in E}$.
  For any edge $e\in E$, define
\begin{align}\label{defpsi}
 \psi(e) = \frac{\displaystyle{\sum_{g\in E: g< e} \delta_g^{-1}}}{w_e^{-1}+\displaystyle{\sum_{g\in E: g< e} \delta_g^{-1}}},
 \end{align}
with the convention that $\psi(e)=1$ if the sum in the numerator is empty, i.e.~if $|e|=1$.
{Note that, roughly speaking, $\psi(e)$ corresponds to the probability that the \Go~{\it restricted} to the path from the root to $e^+$ hits $e^+$ before  returning to the root, after having reached $e^-$. This interpretation in terms of one-dimensional ruin probabilities will be made rigorous at the end of Section \ref{sec:ext}.}\\
 Finally, let us define, for any $e\in E$,
 \begin{align}\label{defPsi}
 \Psi(e)=\prod_{g\le e} \psi(g).
 \end{align}
Recall that we defined, just before \eqref{defBRNS}, the set  $\Pi$ of all cutsets of the tree $\mathcal{G}$. In the statement of the next theorem, we will assume that the following technical condition on ${\bf X}$ holds:
\begin{align}\label{maincond}
\exists M\in(1,\infty)\text{ s.t.~}\frac{1}{M}\le\frac{\sum_{g\le e}1/\delta_g}{\sum_{g\le e}1/w_g}\le M,\text{ for all }e\in E.
\end{align}
The recurrence or transience of ${\bf X}$ on the tree $\mathcal{G}$ is going to be characterized by the quantity
\begin{align}\label{defBR}
RT(\mathcal{G},\X)\Def \sup\left\{\lambda>0: \inf_{\pi\in \Pi}\sum_{e\in\pi}(\Psi(e))^{\lambda}>0\right\}.
\end{align}
One can easily check using \eqref{defBRNS} that the branching-ruin number $br_r(\mathcal{G})$ of $\mathcal{G}$ is equal to $RT(\mathcal{G},{\bf S})$ where ${\bf S}$ is the simple random walk (i.e.~$w_e=\delta_e=1$ for every edge $e$). Therefore, the quantity $RT(\cdot,\cdot)$ can be seen as a generalized version of the branching-ruin number.\\
The next result provides a sharp and effective criterion for recurrence/transience of \Go s, under the condition \eqref{maincond}.

\begin{theorem}\label{mainth}  Consider a  \Go~${\bf X}$ defined on an infinite,  locally finite, tree $\mathcal{G}$. If $RT(\mathcal{G},\X)<1$ then $\mathbf{X}$ is recurrent. If $RT(\mathcal{G},\X)>1$ and if \eqref{maincond} is satisfied then $\mathbf{X}$ is transient.
\end{theorem}

Let us comment condition \eqref{maincond}. First, \eqref{maincond} is satisfied by any multiplicative {\scriptsize ORRW} with general initial weights, i.e. $w_e\in(0,\infty)$ and $\delta_e=\delta \times w_e$ for any $e\in E$ and for some parameter $\delta\in(0,\infty)$, in which case the ratio in \eqref{maincond} is always equal to $\delta$. This includes the case of Markov chains, by choosing $\delta=1$. Note that \eqref{maincond} allows for more inhomogeneity than these cases.\\
Second, it should be noted that, in fact, this condition is essentially necessary if one wants to follow the strategy we adopt here. Indeed, it is not too difficult to find a counterexample to  Lemma \ref{lemmaquasi} when  \eqref{maincond} does not hold. Here, we choose to give \eqref{maincond} as condition, because it is easy to check for any model, but it should be noted that  Theorem~\ref{mainth} still holds if we replace \eqref{maincond} by {\it quasi-independence} as described in Lemma \ref{lemmaquasi}. Besides, we believe that  Theorem \ref{mainth} fails without assuming quasi-independence.\\

In most cases the quantity $RT(\mathcal{G},\X)$ can be explicitly computed.  Let us consider a general example.
Fix  a tree $\mathcal{G}$ such that $br(\mathcal{G})>1$. A  process $\X$ is a biased {\scriptsize ORRW} with parameter $\delta\in(0,\infty)$ if it is a \Go~with initial weights $w_e=\beta^{-|e|}$ and reinforced weights $\delta_e=\delta\times\beta^{-|e|}$ for every edge $e\in E$. The case $\beta>1$ corresponds to a bias towards the root and the case $\beta\in(0,1)$ corresponds to an outward bias. The next result generalizes Corollary 1.5 in \cite{CHK}. Note that the case $\delta=1$ corresponds to a usual biased random walk, and the case $\beta=1$ corresponds to \Or.

\begin{corollary}\label{corbiased}
Let  ${\bf X}$ be  a biased  {\scriptsize ORRW}  as described above on a tree $\mathcal{G}$ with $br(\mathcal{G})>1$.  First, if $\beta\in(0,1]$, then $RT(\mathcal{G},{\bf X})=\infty$ and thus ${\bf X}$ is transient for any parameter $\delta>0$.
Second, if $\beta>1$, we have that 
\[
RT(\mathcal{G},{\bf X})=\frac{\ln\left(br(\mathcal{G})\right)}{\ln\left(\delta(\beta-1)+1\right)}.
\]
In particular, ${\bf X}$ is transient if $\delta<(br(\mathcal{G})-1)/(\beta-1)$ and it is recurrent if $\delta>(br(\mathcal{G})-1)/(\beta-1)$.
\end{corollary}

 \begin{remark}\label{excited}
As explained in the introduction, we believe that our techniques can be used for different models. In particular, it should be possible to apply those to excited random walks on trees. It should be noted that, as a first step, it is quite straightforward to apply the techniques to the  {\it $M$-digging random walk}, an extreme case of the excited random walk introduced in \cite{Volk} and \cite{BasdSingh}. This would provide new results about this model on general trees.
 \end{remark}

\begin{remark}\label{remrwre}
It is possible to implement these techniques in order to study random walk in random environment (\Rw). We obtain, in a separate work (in progress), criterion for the recurrence/transience of \Rw~when the environment is not independent and under some general assumption, generalizing \cite{CollBarb}. For random walks in independent random environment, we believe that our techniques can be pushed to study the critical phases of \Rw, left open in \cite{LyoPem}. Finally, it should be noted that one of the critical cases was studied in \cite{PP}, for i.i.d.~and balanced environments. Their results can be rephrased as follows: on a tree $\mathcal{G}$, if the branching-ruin number is such that $br_r(\mathcal{G})>1/2$ then  the \Rw~is transient and if $br_r(\mathcal{G})<1/2$ then it is recurrent.
\end{remark}

\section{Features of the Branching-Ruin Number} \label{features}

In this Section, we explore different aspects of the branching-ruin number. First, we relate it to the growth of polynomial trees. Second, we propose a construction in order to provide a polynomial counterpart of Galton-Watson trees and show how the branching-ruin number naturally appears in the structure of these random trees. Third, we express the number $RT(\cdot,\cdot)$, defined in \eqref{defBR}, and in particular the  branching-ruin number in terms of the Hausdorff dimension of the boundary of the tree at infinity with respect to a particular metric.

\subsection{Growth of polynomial trees} \label{growth}

As highlighted in the introduction, the branching-ruin number of a  tree, see \eqref{defBRNS}, appears to be a nice way to measure polynomial trees.  For a tree $\mathcal{G}$, we define
the {\it polynomial growth} of the tree as
\begin{align*}
Pgr(\mathcal{G})= \sup\left\{\lambda>0: \liminf_{n\to \infty}\sum_{e\in E_n}n^{-\lambda}>0\right\}=\liminf_{n\to\infty} \frac{\ln\left(\left|E_n\right|\right)}{\ln\left(n\right)},
\end{align*}
where $E_n=\{e\in E: \ |e|=n\}$ is the set of edges at generation $n$.\\

By comparing it to \eqref{defBRNS}, it is easy to see that $br_r(\mathcal{G})\le Pgr(\mathcal{G})$, as the sets $E_n$ are particular choices of cutsets. In general, these two numbers may not be equal, and one can easily find examples where they indeed differ (e.g.~build a polynomial tree with a structure similar to the second example p.936 of \cite{L90}). Nevertheless, one can prove that if $\mathcal{G}$ is {\it spherically symmetric} (i.e.~if the degree of a vertex depends only on its generation) then $br_r(\mathcal{G})= Pgr(\mathcal{G})$.\\

In particular, if a tree $\mathcal{G}$ is  spherically symmetric  and if  $\left|E_n\right|\times n^{-a}$ is asymptotically bounded away from $0$ and the infinity for some $a\in(0,\infty)$, then $br_r(\mathcal{G})=a$.

\subsection{Generating random polynomial trees} \label{randomtree}
In this Section, we consider a natural way to generate random polynomial trees and we show how the branching-ruin number arises naturally from the structure of the tree. This is similar to the fact that the branching number of an infinite supercritical Galton-Watson tree is a.s.~equal to its mean offspring (see \cite{L90}).\\

We do not work with the most general way to generate polynomial trees, but we use a construction that looks to  be  an interesting polynomial counterpart to Galton-Watson trees. As for the latter, the law of the random trees we consider depends only on one probability distribution and its behavior (i.e.~if it is infinite with positive probability or not) depends only on the mean of this distribution. The general idea of this construction uses the fact that,  along any infinite ray of a polynomial tree, most of vertices have only one child and, more and more rarely (logarithmically often), a vertex behaves differently and has  several children or no child. Hence, a typical ray in an infinite polynomial tree looks most of the time like a line where vertices have only one child, plus some rare vertices with several children, providing the tree structure. Our construction also allows for leaves in the tree. The tree we propose can be seen as a Galton-Watson tree where each edge is replaced by a random number of edges in series (depending on the height).\\
Interestingly, the branching-ruin number  turns out to be the natural parameter for this random tree, that is the mean of the distribution mentioned above.\\

Let us construct this polynomial random tree. Start by fixing a collection of nonnegative real numbers $(p_k)_{k\ge-1}$ such that $\sum_{k\ge-1}p_k=1$ and $p_{-1}\neq1$. Let $L$ be an integer-valued random variable which is equal to $k$ with probability $p_k$, for any integer $k\ge-1$. This generic random variable will be used to define the offspring distributions in the tree. Assume that ${\bf E}[L^2]=:\sigma^2<\infty$ and define $m={\bf E}[L]\in(-1,\infty)$.\\
For each $n\ge1$, let $\epsilon_n$ be a random variable taking values in $\{0,1\}$ and defined by ${\bf P}(\epsilon_n=1)=1/n=1-{\bf P}(\epsilon_n=0)$.\\
Now construct a random tree $\mathcal{T}_m$ iteratively,  starting with one vertex at level $1$ and such that each vertex $x$ at level $n\ge1$ has $Z_{n}^{(x)}$ offsprings in the tree, where $Z_{n}^{(x)}=1+\epsilon_{n}^{(x)}L^{(x)}$ with $\epsilon_{n}^{(x)}$ and $L^{(x)}$ being independent copies of $\epsilon_{n}$ and $L$, respectively, and are  independent of everything else.\\

For this random tree $\mathcal{T}_m$, a vertex at generation $n$ has an average number of offspring equal to $1+\frac{m}{n}$. Then, it is natural to expect that this tree is infinite with positive probability if and only if $m>0$, see Proposition \ref{siham} below.\\
One could argue that the law of $(\epsilon_n)$ is arbitrary, but one should be convinced that it is essentially the only good choice by the following arguments.
First, if $m>0$, the average number of vertices in the $n$-th generation of $\mathcal{T}_m$ is of the order of $n^m$ and $\mathcal{T}_m$ is indeed a polynomial tree. Second, if $\epsilon_n$ was equal to $1$ with probability $1/n^a$ with $a\in(0,1)$ (resp.~$a>1$), then we would obtain that the size of the generations behaves like a stretched exponential (resp.~converges to a finite quantity). Hence, choosing $a=1$ is indeed the natural feasible choice in order to obtain a tree with polynomial growth.\\

The following result again justifies our statement that the branching-ruin number is a good way to measure polynomial trees.
\begin{proposition} \label{siham}
Let $\mathcal{T}_m$ be a random polynomial tree constructed as above. First, $\mathcal{T}_m$ is infinite with positive probability if and only if $m>0$. Second, if $m>0$ and on the event that $\mathcal{T}_m$ is infinite, we have that $br_r(\mathcal{T}_m)=m$ almost surely.
\end{proposition}
\begin{proof}
The first statement is easy to prove by the following observation. Let $\mathcal{T}_m$ be a random tree as described above and apply the following procedure. For any vertex $x\in\mathcal{T}_m$, if $\epsilon^{(x)}_{|x|}=0$, then we remove $x$ from the tree (together with its incident edges) and add an edge between the father of $x$ and the unique offspring of $x$; otherwise, if $\epsilon^{(x)}_{|x|}=1$, we keep $x$ as it is. The tree obtained in this manner is simply a Galton-Watson tree with offspring distribution given by that of $1+L$, and this new tree is infinite if and only if $\mathcal{T}_m$ is infinite. Hence, $\mathcal{T}_m$ is infinite with positive probability if and only if $1+m>1$, which proves the first statement.\\

Let us now prove the second statement of the Proposition. We mimic a simple argument from \cite{L90}. Let us consider the percolation on $\mathcal{T}_m$ where each edge $e$ at level $n$ is open with probability $1-\delta/n$ for some $\delta>0$ (forcing the edge to be open as long as $\delta>n$).\\
On one hand, we claim that the cluster of the root is infinite with positive probability if $\delta<br_r(\mathcal{T}_m)$  and it is a.s.~finite if  $\delta>br_r(\mathcal{T}_m)$. First, if we let $\X$ be a \Go~satisfying, for $e\in E$,  $\psi(e)=1-\delta/|e|$ if $|e|>\delta$ and $\psi(e)=1$ otherwise, one can easily compute, using \eqref{defBRNS} and \eqref{defBR}, that $RT(\mathcal{G},\X)=br_r(\mathcal{G})/\delta$ (a similar computation is done in the proof of Theorem \ref{corpoly}). Second,  by Remark \ref{remperco} and Theorem \ref{corpoly},  fixing $\mathcal{T}_m$ on the event that it is infinite, then the cluster of the root is infinite with positive probability if $\delta<br_r(\mathcal{T}_m)$  and it is a.s.~finite if  $\delta>br_r(\mathcal{T}_m)$.\\
On the other hand, this percolation simply defines a random subtree $\mathcal{T}_{\text{perc}}$ of the random tree $\mathcal{T}_m$. Each vertex at level $n$ in the subtree has an average number of offsprings equal to $(1-\delta/n)(1+m/n)= 1+(m-\delta)n^{-1}-\delta mn^{-2}$. 
Let us prove that $\mathcal{T}_{\text{perc}}$ is infinite with positive probability if and only if $m-\delta>0$. This would imply that $br_r(\mathcal{T}_m)=m$ and conclude the proof.\\
Let $V_n=\{v\in V:\ |v|=n\}$, for any $n\ge0$, be the set of vertices at generation $n$ of $\mathcal{T}_{\text{perc}}$. Note that $V_n$ is random. Let $\widetilde{Z}_j$ be the offspring distribution of a vertex at generation $j$ in $\mathcal{T}_{\text{perc}}$. Let $\mathcal{G}_n=\sigma\left(V_0,\dots,V_n\right)$ be the filtration generated by all the information contained in the $n+1$ first generations of the tree. One can easily see that, for any $n\ge0$,
\[
{\bf E}\left[\left.\left|V_{n+1}\right|\right|\mathcal{G}_n\right]=\left|V_n\right|\times \left(1+\frac{m-\delta}{n}-\frac{\delta m}{n^{2}}\right).
\]
If $m-\delta\le0$, $\left(\left|V_n\right|\right)_n$ is a nonnegative super-martingale and thus converges to $0$ almost surely.\\
Now, assume that $m-\delta>0$. Theorem 1 of \cite{BP}, see the upper-bound of (2.4) therein, states that
\[
\lim_{n\to\infty}{\bf P}\left(\left|V_n\right|>0\right)\ge \limsup_{n\to\infty} \left[  {\bf E}\left[\left|V_n\right|\right]^{-1}+\sum_{j=1}^{n}\frac{{\bf E}\left[\widetilde{Z}_{j}^2\right]-{\bf E}\left[\widetilde{Z}_{j}\right]}{{\bf E}\left[\widetilde{Z}_{j}\right]}    {\bf E}\left[\left|V_{j}\right|\right]^{-1}  \right]^{-1}.
\]
One can easily compute from the definitions that, for any $j\ge1$,
\[
{\bf E}\left[\widetilde{Z}_{j}^2\right]-{\bf E}\left[\widetilde{Z}_{j}\right]\le \frac{m+\sigma^2}{n}\text{ and } {\bf E}\left[\left|V_{j}\right|\right]\ge cj^{\frac{m-\delta}{2}},
\]
for some constant $c>0$. Hence, as $m-\delta>0$, we obtain that,
\[
{\bf P}\left(\mathcal{T}_{\text{perc}}\text{ is infinite}\right)=\lim_{n\to\infty}{\bf P}\left(\left|V_n\right|>0\right)\ge \lim_{n\to\infty} \frac{1}{1+\sum_{j=1}^{n}\frac{m+\sigma^2}{j}\times \frac{c}{j^{\frac{m-\delta}{2}}}}>0.
\]
Hence, $\mathcal{T}_{\text{perc}}$ is infinite with positive probability if and only if $m-\delta>0$. Recall that we have already proved that  if $br_r(\mathcal{T}_m)-\delta>0$ (resp.~if $br_r(\mathcal{T}_m)-\delta<0$) then $\mathcal{T}_{\text{perc}}$ is infinite with positive probability (resp.~finite a.s.),  therefore we can conclude that $m=br_r(\mathcal{T}_m)$.

\hfill
\end{proof}

\subsection{Hausdorff dimension}\label{sec:Haus}

Here, we prove that the quantity $RT(\cdot,\cdot)$, defined in \eqref{defBR}, and in particular  the branching-ruin number, can be rephrased as the Hausdorff dimension of the boundary of the tree at infinity, with respect to a particular distance.\\
Let us recall  the definition of the Hausdorff dimension of the boundary of an infinite tree as Furstenberg \cite{Furs} defined it, see also \cite{LP}. 
First, the {\it boundary} $\partial\mathcal{G}$ of the tree at infinity is defined as the set of infinite rays, that is the set of all infinite simple paths started from the root (in particular this boundary does not consider the leaves). For an infinite ray $\xi\in\partial\mathcal{G}$, we denote $\xi_n$ the edge of $\xi$ at generation $n$.
A natural metric on $\partial\mathcal{G}$ is the following: if $\xi,\eta\in\partial\mathcal{G}$ have exactly $n$ edges in common, then $d(\xi,\eta)=\exp({-n})$. In particular, for $e\in E$, if we let
\begin{align}\label{Be}
B_e=\left\{\xi\in\partial\mathcal{G}: \xi_{|e|}=e\right\},
\end{align}
then the diameter of $B_e$ is
\[
\mathrm{diam}\ B_e= \min\left\{\exp({-n}): \forall \xi,\eta\in B_e,\ \xi_n=\eta_n\right\}.
\]
Thus, we have that $\mathrm{diam}\ B_e \le \exp\{-|e|\}$ and equality holds if and only if $e^+$ has at least two children in the tree.  A collection $\mathscr{C}$ of subsets of $\partial\mathcal{G}$ is said to be a {\it cover} if
\[
\bigcup_{B\in\mathscr{C}}B=\partial\mathcal{G}.
\]
The Hausdorff dimension of $\partial\mathcal{G}$ is defined as
\[
\mathrm{dim}_{\mathfrak{H}}\ \partial\mathcal{G}=\sup\left\{\lambda>0: \inf_{\mathscr{C}\text{ countable cover}}\sum_{B\in\mathscr{C}}\left(\mathrm{diam}\ B\right)^\lambda>0\right\},
\]
which is also equal to
\[
\mathrm{dim}_{\mathfrak{H}}\ \partial\mathcal{G}=\sup\left\{\lambda>0: \inf_{\pi\in \Pi}\sum_{e\in\pi}\exp({-\lambda |e|})>0\right\}.
\]
This last quantity is simply the natural logarithm of the branching number defined as, by \eqref{defbranch}, we have
\begin{align*}
br(\mathcal{G})=\exp\left( \mathrm{dim}_{\mathfrak{H}}\ \partial\mathcal{G}\right).
\end{align*}

We are now going to define the Hausdorff dimension of the boundary of the tree in a metric induced by the ruin probabilities of a \Go~along the rays of the tree.\\
First, let us restrict ourselves to the case where the quantity $\Psi$ defined in \eqref{defPsi} goes to $0$ along any infinite ray. More precisely, for $\xi\in\partial\mathcal{G}$,  we assume that
\begin{align}\label{assumption}
\lim_{n\to\infty}\Psi(\xi_n)=0.
\end{align}
This assumption simply ensures that $\Psi$ induces a metric on the infinite rays. Recall also that $\Psi$ is decreasing to $0$ along any  ray.\\

Now, let us define the following distance on $\partial\mathcal{G}$: for $\xi,\eta\in\partial\mathcal{G}$, if $e$ is their common edge with highest generation, then $d_\Psi(\xi,\eta)=\Psi(e)$. 
The assumption \eqref{assumption} ensures that $d_\Psi(\xi,\xi)=0$ for any $\xi\in\partial\mathcal{G}$. In particular, for $e\in E$, defining $B_e$ as in \eqref{Be}, we can compute the diameter with respect to $d_\Psi$ to be
\[
\mathrm{diam}_\Psi  B_e= \min\left\{\Psi(g): g\in \xi,\ \forall \xi\in B_e \right\}.
\]
Finally, define the $\Psi$-Hausdorff dimension of $\partial\mathcal{G}$ as
\begin{align*}
\mathrm{dim}^\Psi_{\mathfrak{H}}\ \partial\mathcal{G}&=\sup\left\{\lambda: \inf_{\mathscr{C}\text{ countable cover}}\sum_{B\in\mathscr{C}}\left(\mathrm{diam}_\Psi\ B\right)^\lambda>0\right\}\\
&=\sup\left\{\lambda: \inf_{\pi\in \Pi}\sum_{e\in\pi}\left(\Psi(e)\right)^\lambda>0\right\}.
\end{align*}
Thus, we have that $RT(\mathcal{G},\X)=\mathrm{dim}^\Psi_{\mathfrak{H}}\ \partial\mathcal{G}$. In particular $br_r(\mathcal{G})$ is equal to the Hausdorff dimension of the boundary of the tree at infinity when we choose that the distance between two infinite rays $\xi,\eta\in\partial\mathcal{G}$ with common edge with highest generation $|e|$ is $d(\xi,\eta)=1/|e|$.

\section{Applications of the Branching-Ruin Number} \label{proofcorsuper}
In this Section, we prove that  Theorem \ref{corpoly}, Corollary \ref{corsuper} and Corollary \ref{corbiased} are simple consequences of Theorem \ref{mainth}.

\begin{proof}[Proof of Theorem \ref{corpoly}]
Recall that we consider a {\scriptsize ORRW} $\X$ with parameter $\delta\in(0,\infty)$ and recall the definitions \eqref{defpsi} of $\psi(\cdot)$ and \eqref{defPsi}  of $\Psi(\cdot)$. In this case, by \eqref{defpsi}, we have that, for any edge $e\in E$, $\psi(e)=(|e|-1)/(|e|-1+\delta)$ if $|e|\ge2$ and $\psi(e)=1$ if $|e|=1$.
Hence, for any $\lambda>0$, there exist constants  $c_0,c_1\in(0,\infty)$ such that, for any $\pi\in\Pi$,
\begin{align*}
\sum_{e\in\pi}(\Psi(e))^{\lambda}&\ge \sum_{e\in\pi}\prod_{n=1}^{|e|} \left(1-\frac{\delta}{\delta+n}\right)^{\lambda}\ge \sum_{e\in\pi} c_0\exp\Big\{-\lambda\delta\sum_{n=1}^{|e|} \frac{1}{\delta+n}\Big\}\\
&\ge \sum_{e\in\pi}c_1\frac{1}{|e|^{\lambda\delta}}.
\end{align*}
Similarly,  for any $\lambda>0$, there exist two constants $c_1,c_2\in(0,\infty)$ such that, for any $\pi\in\Pi$,
\begin{align*}
c_1\sum_{e\in\pi}\frac{1}{|e|^{\lambda\delta}}\le \sum_{e\in\pi}(\Psi(e))^{\lambda}\le c_2\sum_{e\in\pi}\frac{1}{|e|^{\lambda\delta}}.
\end{align*}
Finally, by comparing \eqref{defBR} and \eqref{defBRNS}, one can see that $RT(\mathcal{G},\X)=br_r(\mathcal{G})/\delta$. Theorem \ref{mainth} easily provides the conclusion.
\hfill
\end{proof}

\begin{proof}[Proof of Corollary \ref{corsuper}]
Here we assume that $br(\mathcal{G})>1$ and we fix $\delta>0$. Therefore, by \eqref{defbranch}, there exists $\varepsilon>0$ and $c>0$ such that
\[
\inf_{\pi\in \Pi}\sum_{e\in\pi}(1+\varepsilon)^{-|e|}>c.
\]
Hence, for any $\lambda>0$, proceeding as in the previous proof, there exist constants  $c_1,c_3,c_4\in(0,\infty)$ such that, for any $\pi\in\Pi$,
\begin{align*}
\sum_{e\in\pi}(\Psi(e))^{\lambda}&\ge \sum_{e\in\pi}c_1\frac{1}{|e|^{\lambda\delta}}\ge c_3\sum_{e\in\pi}(1+\varepsilon)^{-|e|}> {c_4}.
\end{align*}
Hence, by definition  \eqref{defBR}, we have that  $RT({\mathcal{G}},\X)>1$ and we can thus conclude by Theorem \ref{mainth} that the walk is transient.\hfill
\end{proof}

\begin{proof}[Proof of Corollary \ref{corbiased}]
We now consider $\X$ to be the biased {\scriptsize ORRW} on a tree $\mathcal{G}$ with $br(\mathcal{G})>1$.
One can prove by straightforward computations that, for any $\beta>1$,  any $\delta>0$ and any $\lambda>0$, there exist constants $c_4,c_5\in(0,\infty)$ such that, for any $\pi\in\Pi$,
\begin{align*}
c_4\sum_{e\in\pi}\left( \frac{1}{\delta(\beta-1)+1}\right)^{\lambda|e|}\le \sum_{e\in\pi}(\Psi(e))^{\lambda}\le c_5 \sum_{e\in\pi}\left( \frac{1}{\delta(\beta-1)+1}\right)^{\lambda|e|}.
\end{align*}
If $\beta=1$, this corresponds to the statement of Corollary \ref{corsuper}. If $\beta\in(0,1)$ and for any $\delta>0$, it is easy to check that $\Psi(e)$ converges to a positive constant as $|e|$ goes to infinity, on any infinite ray, and therefore $RT(\mathcal{G},{\bf X})=\infty$, for any $\delta>0$.\\
If $\beta>1$, using the definition \eqref{defbranch} of the branching number, the definition \eqref{defBR} of $RT(\cdot,\cdot)$ and by a simple computation, we have that
\[
RT(\mathcal{G},{\bf X})=\frac{\ln\left(br(\mathcal{G})\right)}{\ln\left(\delta(\beta-1)+1\right)}.
\]
One can then conclude about the recurrence/transience of $\X$ by applying Theorem \ref{mainth}.
\hfill
\end{proof}

 \section{Extensions}\label{sec:ext}
 
Here, we define the same construction as in \cite{CHK} which is a particular case of Rubin's construction. This will allow us to emphasize useful independence properties of the walk on disjoint subsets of the tree.\\

Let $(\Omega, \mathcal{F},\bP)$ denote a probability space on which
\begin{align}\label{defY}
{\bf Y}=(Y(\nu,\mu,k): (\nu,\mu)\in V^2, \mbox{with }\nu \sim \mu, \textrm{ and }k \in \N)
\end{align}
is a family of independent  exponential random variables with mean 1, and where $(\nu,\mu)$ {denotes} an {\it ordered} pair { of vertices}. Below, we use these collections of random variables to generate the steps of $\X$. Moreover,   we  define  a {\it family} of coupled walks using the same collection  of \lq clocks\rq\  $ {\bf Y}$.

Define, for any integer $j\ge0$ and any $\nu,\mu\in V$ with $\nu\sim \mu$, the quantities
\begin{align} \label{wj1}
r(\nu,\mu,j)&=w_{[\nu,\mu]}\1_{\{j=0,\nu<\mu\}}+\delta_{[\nu,\nu_i]}\1_{\{j\ge1\}\cup\{\mu<\nu\}}.
\end{align}

As it was done in \cite{CHK}, we are now going to define a family of coupled processes on the subtrees of $\mathcal{G}$. For any rooted subtree $\mathcal{G}'$ of $\mathcal{G}$, Let us define the {\it extension} $\X^{ (\mathcal{G}')}=(V',E')$  on $\mathcal{G}'$ as follows. Let   the root $\r'$ of $\mathcal{G}'$ be defined as the vertex of $V'$ with smallest distance to $\r$.
For  a collection of nonnegative integers $\bar{k}=(k_\mu)_{\mu: [\nu,\mu]\in E'} $, let 
\[
A^{ (\mathcal{G}')}_{\bar{k},n,\nu}=\{X^{ (\mathcal{G}')}_n = \nu\}\cap\bigcap_{\mu: [\nu,\mu]\in E'} \{\#\{1\le j \le n \colon (X^{ (\mathcal{G}')}_{j-1},X^{ (\mathcal{G}')}_j) = (\nu,\mu)\} = k_\mu\}.
\]
Note that the event $A^{ (\mathcal{G}')}_{\bar{k},n,\nu}$ deals with jumps along oriented edges.\\
Set $\X^{ (\mathcal{G}')}_0=\r'$ and, for $\nu$, $\nu'$ such that $[\nu, \nu']\in E'$ and for $n\ge0$, on the event 
\begin{align}\label{ursula}
A^{ (\mathcal{G}')}_{\bar{k},n,\nu}\cap \left\{\nu' = \argmin_{\mu: [\nu,\mu]\in E'}\Big\{\sum_{i=0}^{k_{\mu}}\frac{Y(\nu, \mu, i)}{r(\nu, \mu,i)} \Big\}\right\}, 
\end{align}
 we set $X^{ (\mathcal{G}')}_{n+1} = \nu'$, where the function $r$ is defined in \eqref{wj1} and the clocks $Y$'s are from the same collection ${\bf Y}$ fixed in \eqref{defY}.\\
 
 We define $\X=\X^{(\mathcal{G})}$ to be the extension on the whole tree.
  It is easy to check, from properties of independent exponential random variables and the memoryless property, that this provides a construction of the \Go~$\X$ on $\mathcal{G}$.\\
This continuous-time embedding is classical: it is called {\it Rubin's construction}, after Herman Rubin  {(see the Appendix in Davis \cite{Dav90})}.\\
Now, if we consider proper subtrees $\mathcal{G}'$ of $\mathcal{G}$, one can check that, with these definitions, the steps of $\X$ on the subtree $\mathcal{G}'$ are given by the steps of $\X^{ (\mathcal{G}')}$ (see \cite{CHK} for details). As it was noticed in \cite{CHK}, for two subtrees $\mathcal{G}'$ and $\mathcal{G}''$ whose edge sets are disjoint, the extensions $\X^{ (\mathcal{G}')}$ and $\X^{ (\mathcal{G}'')}$ are independent as they are defined by two disjoint sub-collections of ${\bf Y}$.

Of particular interest will be the case where $\mathcal{G}'=[\r,\nu]$ is the unique self-avoiding path connecting $\r$ to $\nu$, for some $\nu\in\mathcal{G}$.  In this case, we write $\X^{(\nu)}$ instead of $\X^{([\r,\nu])}$, and we denote $T^{(\nu)}(\cdot)$ the return times associated to $\X^{(\nu)}$. For simplicity, we will also write  $\X^{(e)}$ and $T^{(e)}(\cdot)$ instead of $\X^{(e^+)}$ and $T^{(e^+)}(\cdot)$ for $e\in E$.
Finally, it should be noted that, for any $e\in E$ and any $g\le e$,
\begin{align}\label{eqhit1}
\psi(g)&=\bP\left(T^{(e)}(g^+)\circ \theta_{T^{(e)}(g^-)}<T^{(e)}({\r})\circ \theta_{T^{(e)}(g^-)}\right),\\ \label{eqhit}
\Psi(e)&=\bP\left(T^{(e)}(e^+)<T^{(e)}({\r})\right),
\end{align}
where $\theta$ is the canonical shift on the trajectories.
 
 \section{Recurrence in Theorem \ref{mainth}: the case ${RT(\mathcal{G}, \X)}<1$} \label{sectrec}
 In this section, we assume that  $RT(\mathcal{G},\X)<1$ and prove recurrence. 
 The first part of Theorem \ref{mainth} is a consequence of  the following proposition, which  is an application of the  first moment method.
 \begin{proposition}\label{proprec}
 If
 \begin{align}\label{almadort}
\inf_{\pi\in \Pi}\sum_{e\in\pi}\Psi(e)=0,
\end{align}
then $\X$ is recurrent.
 \end{proposition}
 \begin{proof}
Here, we assume that \eqref{almadort} holds
and that there exists a sequence of cutsets $(\pi_n)_{\n\ge0}\subset\Pi$ such that $\sum_{e \in\pi_n}(\Psi(e))\le \exp({-n})$.\\
We want to estimate the probability that $\X$ escapes to infinity from $\r$, i.e.~never returns to ${\r}$. This requires that $\X$ jumps through at least one edge of each cutset $\pi_n$ before returning to  ${\r}$.\\
 First, fix some edge  $e\in E$ and recall the definition of the extension $\X^{\ssup e}$ from Section \ref{sec:ext}.\\
Using \eqref{eqhit}, we have that
\begin{align*}
&\bP\left(\bigcup_{e\in \pi_n} \left\{T(e^+)<T( {\r})\right\}\right)\le \sum_{e \in \pi_n} \bP\left(T(e^+)<T( {\r})\right)\\
&\le\sum_{e \in \pi_n} \bP\left(T^{(e)}(e^+)<T^{\ssup e}( {\r})\right)=\sum_{e \in \pi_n}\Psi(e)\le {\exp\{-n\}}.
\end{align*}
As this last quantity is summable, the events $\bigcup_{e\in \pi_n} \left\{T(e^+)<T( {\r})\right\}$, $n\ge 0$, happen only finitely often by Borel-Cantelli Lemma, and therefore
\[
\bP(T( {\r})=\infty)\le \bP\left(\bigcap_{n\ge0}\; \;\bigcup_{e\in \pi_n} \left\{T(e^+)<T( {\r})\right\}\right)=0.
\]
This concludes the proof that $\X$ is recurrent.\hfill
\end{proof}

\section{Link with percolation} \label{sec:linkperc}

We are now going to interpret the set of edges crossed before returning to ${\r}$ as the cluster of some correlated percolation and give a stochastic lower-bound to it in terms of  a cluster in  a certain {\it quasi-independent} percolation (see the definition in Lemma \ref{lemmaquasi}). \\

Denote by $\C(\r)$ the set of edges which are crossed by $\X$ before returning to ${\r}$, that is
\[
\C(\r)=\left\{ e \in E \colon T(e^+)<T({\r})\right\}.
\]
This set can  be seen as the cluster containing $\r$ in some correlated percolation.  Next we consider  a different correlated percolation which will be more convenient to us.
Recall Rubin's construction and the extensions introduced in Section \ref{sec:ext}. Then define
\[
\C_{\mathrm{CP}}(\r)= \left\{e \in E:T^{\ssup e}(e^+)<T^{\ssup v}({\r})\right\},
\]
where $T^{\ssup e}(\cdot)$ is defined right before \eqref{eqhit1}. This  defines  a correlated percolation in which an edge $e\in E$ is open if and only if $e\in \C_{\mathrm{CP}}(\r)$. As this percolation is defined using the same extensions as for $\X$, we keep the notation $\bP$ for its measure.
In this context, extensions are useful because, in order to know whether an edge $e$ is open or not, we get rid of the technical complications due to the events on which $\X$ escapes to infinity before  either hitting $e^+$ or returning to ${\r}$. Nevertheless, note that this percolation still has correlation at any length. In fact, in order to determine if two given edges are open or not we need to observe the behaviour of  coupled pair of extensions.\\
In our first result, we relate $\C_{\mathrm{CP}}(\r)$ to $\C(\r)$.
\begin{lemma}\label{lemperc1}
We have that
\[
\bP\left(T({\r})=\infty\right)= \bP\left(\left|\C(\r)\right|=\infty\right) = \bP\left(\left|\C_{\mathrm{CP}}(\r)\right|=\infty\right).
\]
\end{lemma}
\begin{proof}
It is easy to see that a.s.~$\left\{\left|\C(\r)\right|=\infty\right\}=\left\{T({\r})=\infty\right\}$. It remains to prove that a.s.~$\left\{\left|\C_{\mathrm{CP}}(\r)\right|=\infty\right\}=\left\{\left|\C(\r)\right|=\infty\right\}$. We split the proof of this into two parts, by showing a double inclusions.
\begin{itemize}
\item If  $\left|\C_{\mathrm{CP}}(\r)\right|=\infty$ then, for any $n\ge0$, there exists an edge $e$ with $|e|=n$ such that $T^{\ssup e}(e^+)<T^{\ssup e}({\r})$. In this case, either $T(e^+)=T({\r})=\infty$, which means that $\X$ escapes to infinity as it cannot stay forever in any bounded subtree, or $T(e^+)<T({\r})$. Either way, $\X$ hits some vertex at level $n$ before returning to ${\r}$, for any $n\ge 0$. This proves that $\left\{\left|\C_{\mathrm{CP}}(\r)\right|=\infty\right\}\subset\left\{\left|\C(\r)\right|=\infty\right\}$ almost surely.\\
\item If $\left|\C(\r)\right|=\infty$, then, for any $n\ge0$, there exists an edge $e$ with $|e|=n$ such that $T(e^+)<T({\r})$ and thus $T^{\ssup e}(e^+)<T^{\ssup e}({\r})$. This proves that $\left\{\left|\C(\r)\right|=\infty\right\}\subset\left\{\left|\C_{\mathrm{CP}}(\r)\right|=\infty\right\}$ almost surely.
\end{itemize}
\hfill
\end{proof}

For simplicity, for a vertex $v\in V$, we write $v \in\C_{\mathrm{CP}}(\r)$ if one of the edges incident to $v$ is in $\C_{\mathrm{CP}}(\r)$. Besides, recall that for two edges $e_1$ and $e_2$, their common ancestor  with highest generation is the vertex denoted $e_1\wedge e_2$.

\begin{lemma}\label{lemmaquasi}
Assume that \eqref{maincond} holds.
The correlated percolation induced by $\C_{\mathrm{CP}}(\r)$ is \emph{quasi-independent}, i.e.~there exists a constant $C_Q\in(0,\infty)$ such that, for any two edges $e_1,e_2\in E$ with common ancestor $e_1\wedge e_2$, we have that
\begin{align*}
\bP\big(\left.e_1,e_2\in\C_{\mathrm{CP}}(\r)\right|e_1\wedge e_2\in \C_{\mathrm{CP}}(\r)\big)\le& C_Q\bP\big(\left.e_1\in\C_{\mathrm{CP}}(\r)\right|e_1\wedge e_2\in \C_{\mathrm{CP}}(\r)\big)\\
&\times \bP\big(\left.e_2\in\C_{\mathrm{CP}}(\r)\right|e_1\wedge e_2\in \C_{\mathrm{CP}}(\r)\big).
\end{align*}
\end{lemma}

\begin{proof}
%
%
Recall the construction of Section \ref{sec:ext}. Note that if $e_1\wedge e_2=\r$, then  the extensions on $[  \r,e_1^+]$ and on $[  \r,e_2^+]$ are independent, as they are defined by two disjoint collections of exponential clocks, and the conclusion of the lemma holds with $C_Q=1$ by independence.\\
Now, assume that $e_1\wedge e_2\neq\r$ and note that the extensions on $[  \r,e_1^+]$ and on $[  \r,e_2^+]$ are dependent as they use the same exponential clocks on the path $[  \r,e_1\wedge e_2]$. Recall the definition of the processes $Y$, from Section  \ref{sec:ext}. Denote by  $e$ the unique edge such that $e^+=e_1\wedge e_2$ and define
\begin{align*}
N(e)&= \left|\left\{0\le n\le T^{\ssup{e}}( {\r})\circ\theta_{T^{\ssup{e}}(e^+)}: (X^{\ssup{e}}_n,X^{\ssup{e}}_{n+1})=(e^+,e^-)\right\}\right|,\\
L(e)&=\sum_{j=0}^{N(e)-1}\frac{Y(e^+, e^-, j)}{\delta_e},
\end{align*}
where $|A|$ denotes the cardinality of a set $A$,  and $\theta$ is the canonical shift on trajectories.
So that $L(e)$ is the time consumed by the clocks attached to the oriented edge $(e^+,e^{-})$ before $\X^{\ssup{e}}$, $\X^{\ssup{e_1}}$ or $\X^{\ssup{e_2}}$ goes back to $ {\r}$ once it has reached $e^+$. Recall that these three extensions are coupled and thus the time $L(e)$ is the same for the three of them.\\
For $i\in\{1,2\}$, let $v_i$ be the vertex which is the offspring of $e^+$ lying the path from $  \r$ to $e_i$. Note that $v_i$ could be equal to $e_i^+$.
As before, let us define, for $i\in \{1,2\}$,
\begin{align*}
{N^*}(e_i)&= \left|\left\{0\le n\le T^{\ssup{e_i}}(e_i^+): (X^{[e^+,{e_i^+}]}_n,X^{[e^+,{e_i^+}]}_{n+1})=(e^+,v_i)\right\}\right|,\\
L^*(e_i)&=\frac{Y(e^+, v_i, 0)}{w_{(e^+, v_i)}}+\sum_{j=1}^{N^*(e_i)-1}\frac{Y(e^+, v_i, j)}{\delta_{(e^+, v_i)}}.
\end{align*}
Here, $L^*(e_i)$, $i\in\{1,2\}$, is the time consumed by the clocks attached to the oriented edge $(e^+,v_i)$ before $\X^{\ssup{e_i}}$, or $\X^{[e^+,e_i^+]}$,   hits $e_i^+$.\\
Note that the three quantities $L(e)$, $L^*(e_1)$ and $L^*(e_2)$ are independent as they are defined by {three disjoint, and hence independent,} sets of exponential random variables $Y(\cdot,\cdot,\cdot)$. Moreover, we have
\[
\left\{e_1,e_2\in\C_{\mathrm{CP}}(\r)\right\}=\{T^{(e)}(e^+)<T^{(e)}( {\r})\}\cap\{L(e)>L^*(e_1)\}\cap \{L(e)>L^*(e_2)\}.
\]
\begin{figure}[h]   
\includegraphics{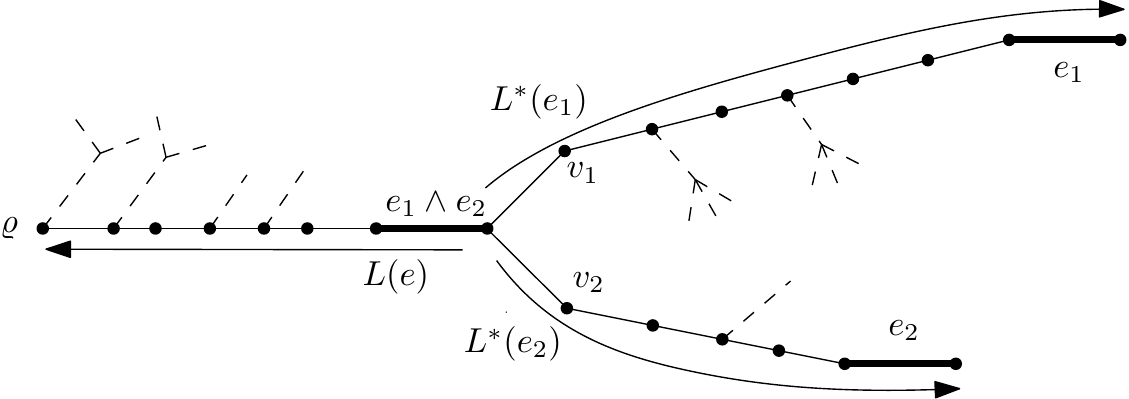}
   \caption{\label{arbre} Representation of $L(e)$, $L^*(e_1)$ and $L^*(e_2)$.}
\end{figure}
Now, note that the random variable $N(e)$ is simply a geometric random variable (counting the number of trials) with success probability
$\delta_e^{-1}/\sum_{g\le e} \delta_g^{-1}$,
and that also holds when conditioned on the event $\{T^{(e)}(e^+)<T^{(e)}( {\r})\}$. Moreover, $N(e)$ is easily seen to be independent of the clocks $Y(e^+,e^-,\cdot)$. Thus, $L(e)$ is simply a geometric sum of i.i.d.~exponential random variables with parameter $\delta_e$. Therefore, $L(e)$ is an exponential random variables with parameter
\begin{align} \label{param}
p:=\frac{1}{\sum_{g\le e} \delta_g^{-1}}.
\end{align}
We cannot draw the same conclusion for $L^*(e_1)$ and $L^*(e_2)$, but we know that they are two continuous random variables as they are a random sum of independent exponential random variables. Let us denote $f_1$ and $f_2$ respectively the densities of $L^*(e_1)$ and $L^*(e_2)$. Then, we have that 
\begin{align}\nonumber
&\bP\left(\left.e_1,e_2\in\C_{\mathrm{CP}}(\r)\right|e_1\wedge e_2\in \C_{\mathrm{CP}}(\r)\right)\\\nonumber
=&\bP\left(L(e)>L^*(e_1) \vee L^*(e_2)\right)\\\nonumber
=&\int_0^\infty\int_0^\infty\int_{x_1\vee x_2}^\infty p\;{\exp\{-pt\}} f_1(x_1)f_2(x_2)dtdx_1dx_2\\\nonumber
=&\int_0^\infty\int_0^\infty {\exp\{-p(x_1\vee x_2)\}} f_1(x_1)f_2(x_2)dx_1dx_2\\\nonumber
\le&\int_0^\infty \int_{0}^\infty {\exp\{-\frac{p}{2}(x_1+x_2)\}} f_1(x_1)f_2(x_2)dx_2dx_1,\\\nonumber
\end{align}
where we used that $x_1\vee x_2\ge (x_1+x_2)/2$.
We can then write the last integral as a product, which yields
\begin{equation} \label{intprod}
\begin{aligned}
&\bP\left(\left.e_1,e_2\in\C_{\mathrm{CP}}(\r)\right|e_1\wedge e_2\in \C_{\mathrm{CP}}(\r)\right)\\
&\le \left(\int_0^\infty {\exp\{-px_1/2\}} f_1(x_1)dx_1\right)\cdot \left(\int_0^\infty {\exp\{-px_2/2\}} f_2(x_2)dx_2\right).
\end{aligned}
\end{equation}
We describe in detail how to treat the first integral appearing in the right-hand side of  \eqref{intprod} in the last product. The way to deal with the  second one is  identical. First, note that
\[
\int_0^\infty {\exp\{-px_1/2\}} f_1(x_1)dx_1=\bP\left(\widetilde{L}(e)>L^*(e_1)\right),
\]
where $\widetilde{L}(e)$ is an exponential variable with parameter $p/2$. Now, given the particular form \eqref{param} of $p$, $\widetilde{L}(e)$ has the same law as $L(e)$ where we replace  the weights $\delta_g$, for $g\le e$ {\it only}, by $\delta_g/2$, $g\le e$ and keep the other weights the same.
Let $\widetilde{\psi}(g)$, for $e<g\le e_1$, have the same definition as $\psi$ but where we replace  the weights $\delta_g$ by $\delta_g/2$ for $g\le e$ {\it only}. First, we obtain
\begin{align*}
\bP\left(\widetilde{L}(e)>L^*(e_1)\right)&=\prod_{e<g\le e_1}\widetilde{\psi}(g)=\prod_{e<g\le e_1}\frac{2p^{-1}+\sum_{e<v<g}\delta_v^{-1}}{2p^{-1}+\sum_{e<v<g}\delta_v^{-1}+w_g^{-1}}\\
&=\bP\left({L}(e)>L^*(e_1)\right)\prod_{e<g\le e_1}\left( 1+\frac{p^{-1}}{p^{-1}+\sum_{e<v<g}\delta_v^{-1}} \right)\\
&\qquad \times\left(1-\frac{p^{-1}}{2p^{-1}+\sum_{e<v<g}\delta_v^{-1}+w_g^{-1}}\right)\\
&=\bP\left({L}(e)>L^*(e_1)\right)\\
&\times\prod_{e<g\le e_1}\left( 1+\frac{p^{-1}w_g^{-1}}{\left(p^{-1}+\sum_{e<v<g}\delta_v^{-1}\right)\left(2p^{-1}+\sum_{e<v<g}\delta_v^{-1}+w_g^{-1}\right)} \right).
\end{align*}
Our goal is to control the last term in the last display.
Recalling that \eqref{maincond} holds for some constant $M\in(1,\infty)$, one can compute
\begin{align*}
& \prod_{e<g\le e_1}\left( 1+\frac{p^{-1}w_g^{-1}}{\left(p^{-1}+\sum_{e<v<g}\delta_v^{-1}\right)\left(2p^{-1}+\sum_{e<v<g}\delta_v^{-1}+w_g^{-1}\right)} \right)\\
&\le \exp\left({p^{-1}}\sum_{e<g\le e_1} \frac{w_g^{-1}}{\left(\sum_{v< g}\delta_v^{-1}\right)\left(\sum_{v< g}\delta_v^{-1}+w_g^{-1}\right)}\right)\\
&\le \exp\left(p^{-1}M^2\sum_{e<g\le e_1} \frac{w_g^{-1}}{\left(\sum_{v\le g}w_v^{-1}\right)\left(\sum_{v< g}w_v^{-1}\right)}\right)\\
&= \exp\left(p^{-1}M^2\sum_{e<g\le e_1} \frac{\sum_{v\le g}w_v^{-1}-\sum_{v< g}w_v^{-1}}{\left(\sum_{v\le g}w_v^{-1}\right)\left(\sum_{v< g}w_v^{-1}\right)}\right)\\
&= \exp\left(p^{-1}M^2\left(\sum_{e\le g< e_1} \frac{1}{\sum_{v\le g}w_v^{-1}}-\sum_{e< g\le e_1} \frac{1}{\sum_{v\le g}w_v^{-1}}\right)\right)\\
&\le \exp\left(p^{-1}M^2\frac{1}{\sum_{v\le e}w_v^{-1}}\right) \le\exp\left(M^3\right).
\end{align*}

We thus have proved that
\[
\int_0^\infty \exp\{-px_1/2\} f_1(x_1)dx_1\le \exp\{M^3\} \bP\left(\left.e_1\in\C_{\mathrm{CP}}(\r)\right|e_1\wedge e_2\in \C_{\mathrm{CP}}(\r)\right).
\]
In the exact same manner, one can prove that 
\[
\int_0^\infty \exp\{-px_2/2\} f_2(x_2)dx_2\le \exp\{M^3\} \bP\left(\left.e_2\in\C_{\mathrm{CP}}(\r)\right|e_1\wedge e_2\in \C_{\mathrm{CP}}(\r)\right).
\]
The two last displays together with \eqref{intprod} provide the conclusion.
\hfill
\end{proof}

\section{Transience in Theorem \ref{mainth}: the case ${RT(\mathcal{G}, \X)}>1$} \label{sec:transience}

First, let us give a bound for the escape probability in terms of some effective conductance. For this purpose, we need to introduce the following modified conductances. Recall the definitions \eqref{defpsi} and \eqref{defPsi} of $\psi(\cdot)$ and $\Psi(\cdot)$, and recall that $\psi(e)=1$ for any edge $e$ such that $|e|=1$, i.e.~$e$ is incident to $\r$.
\begin{definition} \label{defmodcond1}
For any edge $e\in E$, let $c(e)=1$  if $|e|=1$ and, if $|e|>1$,   define 
\begin{align}\label{modcond}
{c}(e) = \frac 1{1- \psi(e)}   \Psi(e).
\end{align}
Define $\Ccal_{\rm eff}$ the effective conductance of $\Gcal$ when  the conductance $c(e)$ is assigned to every edge $e\in E$.  For a definition of effective conductance, see \cite{LP} page 27.
\end{definition}

Recall that $T( {\r})$ be the first time ${\bf X}$ returns to $ {\r}$, i.e.~$T( {\r})=\inf\{n>0:X_n=  \r\}$.

\begin{proposition}\label{proplowerboundperco}
Let ${\bf X}$ be a \Go, as defined in Section \ref{sectionmodel}, with parameters $(\delta_e,w_e)_{e\in E}$ on some tree $\mathcal{G}$. If \eqref{maincond} holds, then there exists $C_Q\in(0,\infty)$ such that
\[
\frac{1}{C_Q}\times\frac{\Ccal_{\rm eff}}{1+\Ccal_{\rm eff}} \le \bP(T( {\r}) = \infty). 
\]
\end{proposition}
\begin{proof}[Proof of Proposition \ref{proplowerboundperco}]
From Lemma \ref{lemperc1} and Lemma~\ref{lemmaquasi}, we can use the lower-bound  in Theorem  5.19 (page 145) of \cite{LP} to obtain the result.\hfill
\end{proof}

Recall that a flow $(\theta_e)$  on a tree is a nonnegative function on $E$ such that, for any $e\in E$, $\theta_e=\sum_{g\in E:g^-=e^+} \theta_g$. A flow is said to be a unit flow if moreover $\sum_{e:|e|=1}\theta_{e}=1$. The following statement is a simple consequence of previous remarks and classical results.

\begin{lemma}\label{simple}
Assume that \eqref{maincond} is satisfied.
Consider the tree $\mathcal{G}$ with the conductances defined in Definition \ref{defmodcond1} and assume that there exists a unit flow $(\theta_e)_{e\in E}$ on $\mathcal{G}$ from $ {\r}$ to infinity which has a finite energy, that is
\[
\sum_{e\in E}\frac{\left(\theta_e\right)^2}{c(e)}<\infty.
\]
Then $\X$ is transient.
\end{lemma}
\begin{proof}
Using Proposition \ref{proplowerboundperco},  if $\C_{\rm eff}>0$ then $\X$ is transient. By Theorem 2.11 (page 39) of \cite{LP},  $\C_{\rm eff}>0$ if and only if there exists a unit flow $(\theta_e)_{e\in E}$ on $\mathcal{G}$ from $ {\r}$ to infinity which has a finite energy.
\hfill
\end{proof}

The following result is inspired by Corollary 4.2 of R.~Lyons \cite{L90}, which is a consequence of the max-flow min-cut Theorem. This result will provide us with a sufficient condition for the existence of a unit flow with finite energy.

\begin{proposition} \label{propLyons}
For any collection of positive numbers $(u_e)_{e\in E}$ such that $\sum_{e:|e|=1}u_{e}=1$ and
\begin{align}\label{condcor}
\inf_{\pi\in\Pi}\sum_{e\in\pi} u_{e}c(e)>0,
\end{align}
there exists  a nonzero flow whose energy is upper-bounded by
\[
\lim_{n\to\infty}\max_{e\in E: |e|=n}\sum_{g\le e} u_g.
\]
\end{proposition}

\begin{proof}
If \eqref{condcor} is satisfied, then the max-flow min-cut Theorem (see \cite{LP}, p.~75) implies that there exists a nonzero flow $(\theta_e)$ satisfying $\theta_e\le u_ec(e)$. Then the energy of this flow is the limit as $n$ goes to infinity of the partial sum
\begin{align*}
\sum_{k=1}^n\sum_{e\in E: |e|=k}\frac{\left(\theta_e\right)^2}{c(e)}\le \sum_{k=1}^n\sum_{e\in E: |e|=k}\theta_eu_e.
\end{align*}
Now, notice that, for any $0\le k\le n$ and any $e\in E$ with $|e|=k$, we have that $\theta_e=\sum_{g:e\le g, |g|=n}\theta_g$ and, moreover, $\sum_{g: |g|=n}\theta_g=\sum_{e:|e|=1}\theta_{e}\le  \sum_{e:|e|=1} u_e c(e) =  1$. Therefore, the energy of this flow $(\theta_e)$ is upper-bounded by
\begin{align*}
\lim_{n\to\infty}\sum_{k=1}^n\sum_{e\in E: |e|=k}\frac{\left(\theta_e\right)^2}{c(e)}\le \lim_{n\to\infty} \sum_{e\in E: |e|=n}\theta_e\sum_{g\le e}u_e\le \lim_{n\to\infty}\max_{e\in E: |e|=n}\sum_{g\le e} u_g.
\end{align*}
\hfill
\end{proof}

\begin{proposition}\label{propfunc}
Fix a real number $\lambda>1$. There exists an absolute constant $C_\lambda<\infty$ such that, for any function $f:\mathbb{N}\to[0,1]$  with $f(0)=1$, we have
\begin{align}\label{claim}
\sum_{n=0}^\infty f(n)\prod_{i=1}^n\left(1-f(i)\right)^{\lambda-1}\le C_\lambda.
\end{align}
\end{proposition}

\begin{proof}
First notice that, for any $n\ge0$,
\begin{align}\label{step1}
f(n)\prod_{i=1}^n\left(1-f(i)\right)^{\lambda-1}\le f(n){\rm e}^{-(\lambda-1)\sum_{i=0}^nf(i)}.
\end{align}
 For any $n\ge0$, we have that
\begin{align*}
&\exp\{-(\lambda-1)\sum_{i=0}^nf(i)\}-\exp\{-(\lambda-1)\sum_{i=0}^{n+1}f(i)\}\\
&=\exp\{-(\lambda-1)\sum_{i=0}^nf(i)\}\left(1-\exp\{-(\lambda-1)f(n+1)\}\right)\\
&\ge \frac{\lambda-1}{3}f(n+1)\exp\{-(\lambda-1)\sum_{i=0}^nf(i)\}\\
&\ge \frac{\lambda-1}{3}f(n+1)\exp\{-(\lambda-1)\sum_{i=0}^{n+1}f(i)\},
\end{align*}
where we have used that $1-{\rm e}^{-x}\ge x/3$ for any $x\in[0,1]$. Together with \eqref{step1}, this implies that
\begin{align*}
\sum_{n=0}^\infty f(n)\prod_{i=1}^n\left(1-f(i)\right)^{\lambda-1}&\le f(0)+\frac{3}{\lambda-1}\sum_{n=0}^\infty \left(e^{-(\lambda-1)\sum_{i=0}^nf(i)}-e^{-(\lambda-1)\sum_{i=0}^{n+1}f(i)}\right)\\
&\le1+ \frac{3}{\lambda-1} \left(e^{-(\lambda-1)}-{\rm e}^{-(\lambda-1)\sum_{i=0}^\infty f(i)}\right).
\end{align*}
This easily implies \eqref{claim} with 
$$C_\lambda= 1+ \frac{3}{\lambda-1}  = \frac{\lambda+2}{\lambda-1}.$$ 
\hfill
\end{proof}

The following result concludes the proof.

\begin{proposition}\label{proptrans}
If $RT(\mathcal{G},\X)>1$ and if \eqref{maincond} is satisfied  then  ${\bf X}$ is transient.
\end{proposition}

\begin{proof}
Fix a real number $\lambda\in \left(1,RT({\mathcal{G}},\X)\right)$ and, for any edge $e\in E$, let us define $u_{e}=1$ if $|e|=1$ and, if $|e|>1$, 
\[
u_e=\left(1-\psi(e)\right)\prod_{g\le e}\left(\psi(g)\right)^{\lambda-1}.
\]
On one hand, we have that, for any $e\in E$,
\begin{align}\label{presque}
\sum_{g\le e}u_{g}\le C_\lambda,
\end{align}
as can be seen by applying Proposition \eqref{propfunc} to functions $f_e$ defined by $f_e(0)=1$ and, for $n\ge1$,  $f_e(n)=1-\psi(g)$ with $g$ the unique edge such that $g\le e$ and $|g|=n\wedge |e|$. We emphasize that \eqref{presque} holds with a uniform bound.\\
On the other hand, using \eqref{modcond}, we have
\begin{align*}
\inf_{\pi\in \Pi}\sum_{e\in\pi} u_{e}c(e) &= \inf_{\pi\in \Pi}\sum_{e\in\pi} \left(\left(1-\psi(e)\right)\left(\Psi(e)\right)^{\lambda-1}\right)\times\frac{\Psi(e)}{1-\psi(e)}\\
&= \inf_{\pi \in \Pi }\sum_{e\in\pi} \left(\Psi(e)\right)^{\lambda}>0.
\end{align*}
Proposition \ref{propLyons} and \eqref{presque} imply that there exists a nonzero flow $(\theta_e)$ whose energy is bounded as
\begin{align*}
\sum_{e\in E}\frac{\left(\theta_{e}\right)^2}{c(e)}\le \lim_{n\to\infty}\max_{e \in E: |e|=n}\sum_{g\le e} u_{g}\le C_\lambda.
\end{align*}
Therefore there exists a unit flow with finite energy and Lemma \ref{simple} implies that  ${\bf X}$ is transient.
\hfill
\end{proof}

\begin{remark}\label{remperco}
Let us emphasize that any independent percolation is quasi-independent. Besides, we can apply  Proposition~\ref{proplowerboundperco} (or alternatively Theorem 5.14  in \cite{LP}) to the independent percolation on $\mathcal{G}$ for which an edge $e\in E$ is open with probability $\psi(e)$. The proof presented in this Section implies that the cluster of the root in this percolation is infinite with positive probability when $RT(\mathcal{G},\X)>1$.\\
Besides, proceeding as in the proof of Proposition \ref{proprec}, one can prove that the cluster of the root in this percolation is a.s.~finite when $RT(\mathcal{G},\X)<1$.\\
Finally, recall that, in the proof of Theorem \ref{corpoly}, we proved that if $\X$ is \Or, then $RT(\mathcal{G},\X)=br_r(\mathcal{G})/\delta$. Hence, the independent percolation in which an edge at level $n+1$ is open with probability $1-\delta/(n+\delta)$ is subcritical if $\delta>br_r(\mathcal{G})$ and supercritical if $\delta<br_r(\mathcal{G})$.
\end{remark}
\section{Critical \Or: proof of Proposition \ref{propcrit}} \label{sectcrit}

Here we  prove Proposition \ref{propcrit} which partially describe the behavior of the \Or~at criticality. In particular, in the following proof, we work with a tree such that $br_r(\mathcal{G})\in(0,\infty)$ and study the \Or~with parameter $\delta_c=br_r(\mathcal{G})$, that is a \Go~with $w_e=1$ and $\delta_e=\delta_c$ for any edge $e\in E$.

\begin{proof}[Proof of Proposition \ref{propcrit}]
The first part about recurrence is in fact a direct consequence of Proposition \ref{proprec}.\\
To prove transience, one has to reproduce the proof of Section \ref{sec:transience} and prove that the effective conductance $\C_{\rm eff}$ of the tree is positive when an edge $e$ is assigned the conductance  specified in \eqref{modcond}, see Proposition \ref{proplowerboundperco}. In the case of \Or, we have  that $c(e)\sim |e|^{1-\delta_c}$, using the fact that $1-\psi(e)= \delta_c/(\delta_c+ |e|-1) \sim \delta_c |e|^{-1}$, for $|e| \ge 2$,  and $\Psi(e)\sim n^{-\delta_c}$.\\
Now, recall that, by assumption, there exists a positive function $f$ such that
\[
\inf_{\pi \in\Pi} \sum_{e\in \pi}\frac{1}{ |e|^{\delta_c}f(|e|)}>0\text{ and } \sum_{n\ge1}\frac{1}{nf(n)}<\infty.
\]
Therefore, we can use Proposition \ref{propLyons} with $u_e=1/(|e|f(|e|))$ and conclude that there exists a nonzero flow with finite energy and thus, by Lemma \ref{simple}, that $\C_{\rm eff}>0$.
\hfill
\end{proof}


 \section{A 0-1 law for recurrence and transience} \label{dichotomy}
 
We prove that  recurrence and transience  for the \Go~satisfy a 0-1 law. 

\begin{proposition}
Let $\X$ be a \Go. 
The event that every vertex (or some vertex) is visited infinitely often happens with probability $0$ or $1$. In particular, this implies that $\X$ is transient if and only if every vertex is visited finitely often, and $\X$ is recurrent if and only if every vertex  is visited infinitely often.
\end{proposition}

\begin{proof}
First, regardless of the current states of the weights and because $\delta_e,w_e\in(0,\infty)$ for any edge $e\in E$, the walk $\X$ goes from one given vertex to another one with a probability lower-bounded by a positive constant  (depending on the choice of the two vertices). Therefore, $\X$ visits one vertex finitely (resp.~infinitely) often if and only if it visits every vertex finitely (resp.~infinitely) often.\\

For any vertex $v\in V\setminus \{\r\}$, let $\mathcal{T}_{v}$ be the subtree consisting of $v^{-1}$, $v$ and all the descendants of $v$. For any $v\in V$, denote $\X^{\mathcal{T}_{v}}$ the extension of $\X$ on the subtree $\mathcal{T}_{v}$, as defined in Section \ref{sec:ext}. Consider the event
\begin{align*}
B
&=\left\{\left|\left\{v\in V\setminus\{\r\}:\left|\{k\ge 0:  X^{\mathcal{T}_{v}}_k=v^{-1}\}\right|<\infty\right\}\right|=\infty\right\}.
\end{align*}
Note that the event $B$ deals with the extensions and not the process ${bf X}$ itself.
We have that, almost surely,
\[
\{T(\r)=\infty\}\subset B \subset\{\X \text{ visits every vertex finitely often}\}.
\]
Indeed, to prove the first inclusion, note that if $T(\r)=\infty$ then infinitely many vertices are ancestors of $X_n$ as soon  as $n$ is large enough. Let us give a short argument to prove the second inclusion. Assume that $\X$ visits one vertex infinitely often and that $B$ holds. Then, $\X$ visits every vertex infinitely often. Besides, as $B$ holds, there exists  a vertex $v$ such that $\left|\{k\ge 0:  X^{\mathcal{T}_{v}}_k=v^{-1}\}\right|=n$, for some finite integer $n$. In this case, if $\X$ visits all the vertices infinitely often, it will eventually jump from $v$ to $\parent v$ for the $n$-th time and come back to $v$. After this time, $\X$ cannot visit $\parent v$ again  and thus it never returns the root, which yields a contradiction.\\
Recall Rubin's construction and the extensions defined in Section  \ref{sec:ext}. In particular, the construction of $\X$ involves a collection of independent and identically distributed (i.i.d.) exponential random variables
\[{\bf Y}=(Y(\nu,\mu,k): (\nu,\mu)\in V^2, \mbox{with }\nu \sim \mu, \textrm{ and }k \in \N).\]
Let us pick these random variables from a given  i.i.d.~collection $(Y_i)_{i\ge0}$, ordered in an arbitrary manner, in the sense that we fix a bijection $f:\mathbb{N}\to\{(\nu,\mu,k):\ \nu,\mu\in V, \nu\sim \mu\text{ and } k\in\mathbb{N}\}$. 

For any $i\in \mathbb{N}$, if $f(i)=(\nu,\mu,j)$, with $\nu,\mu\in V$, $\nu\sim \mu$, $j\in\mathbb{N}$, then we define $||f(i)||=|\nu|$.\\
We claim that the event $B$ is a tail event for the $\sigma$-algebra generated by the sequence $(Y_i)$. First, for any $n\in \mathbb{N}$, we have
\begin{align}\label{ahhhh}
B =\left\{\left|\left\{v\in V:\ |v|>n, \ \left|\{k\ge 0:  X^{\mathcal{T}_{v}}_k=v^{-1}\}\right|<\infty\right\}\right|=\infty\right\}.
\end{align}
Second, for any $n\in\mathbb{N}$, define
\begin{align*}
k(n)&=\max\{i\ge0:\  ||f(i)||\le n\}.
\end{align*}
In words, $k(n)$ is the greatest index such that all the random variables $Y_i$, $i\le k(n)$, are all assigned to vertices at generation less than $n$. In particular, any step performed by $\X$ (or its extensions) from a vertex at generation strictly greater than $n$ does not depend on $Y_i$, $i\le k(n)$. It is straightforward to see that $k(n)$ goes to infinity as $n$ goes to infinity and that $f$ can easily be chosen such that $k(n)$ is finite for every $n$. Indeed, for instance, start by attributing random variables $Y$'s for the first crossing of oriented edges at generation $1$; then assume that for any $i\le n$, oriented edges at generation $i$ have been attributed random variables $Y$'s for their first $n+1-i$ crossings; finally, for any $i\le n+1$, attribute random variables $Y$ to oriented edges at generation $i$ for their first $(n+2-i)$-th crossing.\\
For any $v\in V$ with $|v|>n$, the event $\left\{\left|\{k\ge 0:  X^{\mathcal{T}_{v}}_k=v^{-1}\}\right|<\infty\right\}$ clearly does not depend on the steps of $\X$ performed from vertices at generation less than $n$.
Then, using \eqref{ahhhh}, we obtain that the event $B$ is measurable with respect to $\sigma\left(Y_i,i\ge k(n)\right)$, for any $n\in\mathbb{N}$.\\
Finally, using Kolmogorov's $0$-$1$ law, we obtain that ${\bf P}(B)\in \{0,1\}$.\\

To conclude, note that, on $B^c$, for any vertex $v\in V$ except (at most) a finite number of them, $\X$ jumps from $v$ to $v^{-1}$ infinitely often and thus every vertex, is visited infinitely often.
\hfill
\end{proof}

\begin{ack}
The authors are grateful to Russell Lyons for precious comments on an earlier version of this work.
\end{ack}

\end{document}